\documentclass[12pt, a4paper]{article}

\usepackage[utf8]{inputenc}
\usepackage{amsmath}
\usepackage{amssymb}
\usepackage{amsthm}
\usepackage{geometry}
\usepackage{authblk} 

\newtheorem{theorem}{Theorem}[section]
\newtheorem{proposition}[theorem]{Proposition}
\newtheorem{corollary}[theorem]{Corollary}

\newcommand{\firstpage}[1]{\setcounter{page}{#1}}

\begin{document}

\title{On Mersenne-Bernoulli and Mersenne-Euler polynomials}

\author{Artatrana Suna and Prasanta Kumar Ray$^{}$%
  \footnote{Corresponding author}\\
  \small Department of Mathematics, Sambalpur University,
  \small Jyoti Vihar, Burla, Odisha 768019, India\\
  \small\texttt{suna.19972702@gmail.com},\quad
  \texttt{prasantamath@suniv.ac.in}}

\firstpage{1}

\maketitle

\begin{abstract}
The sequence of Mersenne numbers $\{M_n\}_{n\geq 0}$ is defined as $M_n = 2^n-1.$ In this study we introduce the Mersenne-Bernoulli and Mersenne-Euler polynomials. Using the generating functions and $M$-calculus we find some identities associated with them. Moreover, we define the corresponding matrices with these polynomials, factorise them and find their inverses.

\vspace{1em}

\end{abstract}


\section{Introduction}
Special polynomials have garnered significant attention due to their remarkable versatility in bridging pure mathematics with applied disciplines. Their structured recurrence relations, orthogonality properties, and elegant generating functions make them powerful tools for solving complex problems in areas such as combinatorics, number theory and mathematical analysis. Recent years have seen surge in exploration of special polynomials in degenerate , probabilistic and hybrid forms \cite{KK1,KK2,KK3,KK4,RB,RC,RK,KK5,ZW}.
Among all special polynomials, the Bernoulli and the Euler polynomials play a central role in mathematics. The sequence of Bernoulli polynomials $\{B_n(x)\}_{n \geq 0}$ and the Euler polynomials $\{E_n(x)\}_{n \geq 0}$ \cite{AH} are defined by their exponential generating functions as
\begin{equation}
\frac{te^{tx}}{e^t-1} = \sum_{n=0}^{\infty} B_n(x) \frac{t^n}{n!}, \hspace{2em}|t| < 2\pi
\end{equation}
and
\begin{equation}
\frac{2e^{tx}}{e^t+1} = \sum_{n=0}^{\infty} E_n(x) \frac{t^n}{n!}, \hspace{2em}|t| < \pi.
\end{equation}
First few terms of these polynomials are 
\[B_0(x) = 1, ~B_1(x) = x-\frac{1}{2}, ~B_2(x) = x^2-x+\frac{1}{6}, ~ B_3(x) = x^3-\frac{3}{2}x^2+\frac{1}{2}, \dots\]
and
\[E_0(x) = 1, ~E_1(x) = x-\frac{1}{2}, ~E_2(x) = x^2-x, ~E_3(x) = x^3-\frac{3}{2}x^2+\frac{1}{4}, \dots\]
They are explicitly expressed as
\begin{equation*}
B_{n}(x) = \sum_{k=0}^{n}\binom{n}{k}B_{k}x^{n-k}, 
\end{equation*}
and 
\begin{equation*}
E_{n}(x) = \sum_{k=0}^{n}\binom{n}{k}E_{k}x^{n-k}, 
\end{equation*}
where $B_k = B_k(0)$ and $E_k = E_k(0)$ are the $k$th Bernoulli and Euler numbers respectively. 
Over the past decades, researchers have worked on  the $F$-analogs of special polynomials and the matrices associated with them \cite{DA,GT,MK,KTK,RUF,TKK,URW}.  Motivated by these works we consider the $M$-calculus consisting of the Mersenne numbers $\{M_n\}_{n\geq0}$( A000043 in the OEIS), which are sequence of integers of the form $M_n = 2^n-1.$ These numbers play a crucial role in the search of large primes. The largest known prime till date is $M_{136279841}.$ The Mersenne numbers satisfy the recursion
\begin{equation*}
M_{n+1} = 2M_{n}+1
\end{equation*}
with initial condition $M_0 = 0.$ It is also a binary linear recurrent sequence with
\begin{equation*}
M_{n+2} = 3M_{n+1}-2M_{n}
\end{equation*}
where $M_0 = 0$ and $M_1 = 1.$ By drawing inspiration from the classical factorial, we define the $M$-factorial by assigning $M_0! = 1$ and $M_n = M_n M_{n-1}!.$  Consequently, we define the $M$-binomial coeffients by 
\[\binom{n}{k}_M = \frac{M_{n}!}{M_k!M_{n-k}!}\]
 for $n \geq k \geq 0$ whereas $\binom{n}{k}_M = 0$ for $k > n.$  It is readily seen that 
 \[\binom{n}{k}_M = \binom{n}{n-k}_M, \binom{n}{0}_M = 1 \text{~and~} \binom{n}{1}_M = M_n.\]
Moreover, $\binom{n}{k}_M$ is always an integer as it satisfies the following recurrence relation
\begin{equation} \label{eq: MB}
\binom{n}{k}_M = 2^k\binom{n-1}{k}_M+\binom{n-1}{k-1}_M.
\end{equation}
Further, if $1 \leq j \leq k \leq i$ we have
\begin{equation} \label{eq: MF1}
\binom{i-1}{k-1}_M\binom{k-1}{j-1}_M = \binom{i-1}{j-1}_M\binom{i-j}{k-j}_M.
\end{equation}
We state the $M$-binomial theorem as 
\[(x+_My)^n = \sum_{k=0}^{n} \binom{n}{k}_M x^k y^{n-k}.\]
It is easily followed that $(x+_My)^n = (y+_Mx)^n,~(x+_My)^0 = 1$ and $(x+_My)^1 = (x+y).$ Using the $M$-factorial we define the $M$-exponential function by 
\begin{equation}
e^{xt}_M = \sum_{n=0}^{\infty}x^n\frac{t^n}{M_n!}.
\end{equation}
It is easily seen that $e^{0}_M = 1$ and $e^{xt}_M e^{yt}_M = e^{(x+_My)t}_M.$ Further, for $A(t) = \sum_{n=0}^{\infty}a_n \frac{t^n}{M_n!}$ and $B(t) = \sum_{n=0}^{\infty}b_n \frac{t^n}{M_n!}$ we have $A(t) = B(t)$ if and only if $a_n = b_n$ for $n \geq 0.$ Moreover,  $A(t)B(t) = C(t),$ where $C(t) = \sum_{n=0}^{\infty}c_n \frac{t^n}{M_n!}$ with $c_n = \sum_{k=0}^{n}\binom{n}{k}_Ma_nb_{n-k}$ for $n \geq 0.$

\section{Mersenne-Bernoulli and Mersenne-Euler Polynomials}
For a non-negative integer $n$ and a complex number $x$ let $B_{n,M}(x)$ and $E_{n,M}(x)$ be the functions defined by the following relations
\begin{equation} \label{egf:1}
\frac{te^{tx}_M}{e^t_M-1} = \sum_{n=0}^{\infty} B_{n,M}(x) \frac{t^n}{M_n!}, \hspace{2em}|t| < \frac{2\pi}{\ln |e_M^1|}
\end{equation}
and
\begin{equation}
\frac{2e^{tx}_M}{e^t_M+1} = \sum_{n=0}^{\infty} E_{n,M}(x) \frac{t^n}{M_n!}, \hspace{2em}|t| < \frac{\pi}{\ln |e_M^1|}.
\end{equation}
The numbers $B_{n,M} = B_{n,M}(0)$ and $E_{n,M} = E_{n,M}(0)$ are called the Mersenne-Bernoulli and Mersenne-Euler numbers respectively.
\begin{proposition}
For any non-negative integer $n$ we have
\begin{equation*}
x^n = \sum_{k=0}^{n}\binom{n}{k}_M\frac{1}{M_{k+1}}B_{n-k,M}(x).
\end{equation*}
\end{proposition}
\begin{proof}
From the generating function we have 
\begin{align*}
e^{xt}_M &= \left(\sum_{n=0}^{\infty}B_{n,M}(x)\frac{t^n}{M_n!}\right)\left(\sum_{l=1}^{\infty}\frac{t^{n-1}}{M_n!}\right)\\
&= \sum_{n=0}^{\infty}\left(\sum_{k=0}^{n}\binom{n}{k}_M\frac{1}{M_{k+1}}B_{n-k,M}(x)\right)\frac{t^n}{M_n!}.
\end{align*}
The proof ends with comparison of coefficients.
\end{proof}
By taking $x =0$ we obtain the following
\begin{equation} \label{BN: 1}
\sum_{k=0}^{n}\binom{n}{k}_M\frac{1}{M_{k+1}}B_{n-k,M} = M_n!\delta_{n,0},
\end{equation}
where $\delta_{n,0}$ is the Kronecker delta.
\begin{theorem} \label{TH: BE}
For any complex $x$ and $y$
\begin{equation*}
B_{n,M}(x+_My) = \sum_{k=0}^{n}\binom{n}{k}_MB_{k,M}(x)y^{n-k}
\end{equation*}
and
\begin{equation*}
E_{n,M}(x+_My) = \sum_{k=0}^{n}\binom{n}{k}_ME_{k,M}(x)y^{n-k}.
\end{equation*}
\end{theorem}
\begin{proof}
From the generating function we have,
\begin{align*}
\sum_{n=0}^{\infty}B_{n,M}(x+_My)\frac{t^n}{M_n!}&= \frac{2e^{(x+_My)t}_M}{e^t_M+1} = \frac{2e^{xt}_M}{e^t_M+1} e^{yt}_M \\
&= \left(\sum_{n=0}^{\infty}B_{n,M}(x)\frac{t^n}{M_n!}\right)\left(\sum_{m=0}^{\infty}y^m\frac{t^m}{M_m!}\right)\\
&= \sum_{n=0}^{\infty}\left(\sum_{k=0}^{n}\binom{n}{k}_MB_{k,M}(x)y^{n-k}\right)\frac{t^n}{M_n!}
\end{align*}
The Euler polynomials follow the same method.
\end{proof}

\begin{theorem}
If $n \geq 1$ then 
\begin{equation*}   \label{BP: 1}
B_{n,M}(x+_M 1) - B_{n,M}(x) = M_n x^{n-1}
\end{equation*}
and 
\begin{equation*}   \label{EP: 1}
E_{n,M}(x+_M1) + E_{n,M}(x) = 2x^n.
\end{equation*}
\end{theorem}
\begin{proof}
By the previous theorem we observe that
\[B_{n,M}(x+_M1) = \sum_{k=0}^{n}\binom{n}{k}_M B_{k,M}(x).\]
Thus, \[\frac{te^{xt}_Me^{t}_M}{e^{t}_M-1} = \sum_{n=0}^{\infty}B_{n,M}(x+_M 1).\]
Moreover,
\[\frac{te^{xt}_Me^{t}_M}{e^{t}_M-1}- \frac{te^{xt}_M}{e^t_M-1} = te^{xt}_M.\]
Taking the series expansion of the above we get
\[\sum_{n=0}^{\infty}(B_{n,M}(x+_M 1) - B_{n,M}(x))\frac{t^n}{M_n!} = \sum_{n=0}^{\infty}x^n \frac{t^{n+1}}{M_n!}.\]
As $B_{0,M}(x+_M1) = B_{0,M}(x) = 1$ and $\sum_{n=0}^{\infty}x^n \frac{t^{n+1}}{M_n!} = \sum_{n=1}^{\infty}M_nx^{n-1}\frac{t^n}{M_n!},$ we obtain
\[B_{n,M}(x+_M 1) - B_{n,M}(x) = M_n x^{n-1}\]
for $n \geq 1.$ 
Similarly, we have
\[\frac{2e^{xt}_Me^t_M}{e^t_M+1}+\frac{2e^{xt}_M}{e^t_M+1} = 2e^{xt}_M.\]
Therefore,
\[\sum_{n=0}^{\infty}(E_{n,M}(x+_M1)+E_{n,M}(x))\frac{t^n}{M_n!} = \sum_{n=0}^{\infty}2x^n\frac{t^n}{M_n!}.\]
Equating the powers of $t^n$ yields the desired result.
\end{proof}
By taking $x=0,$  we see that
\begin{equation} \label{eq:Bn1}
B_{n,M}(1) = B_{n,M}, \text{~for~} n \geq 2 
\end{equation}
whereas 
\begin{equation} \label{eq:En1}
E_{n,M}(1) = -E_{n,M}, \text{~for~} n \geq 0.
\end{equation} 
\begin{corollary}
If $n\geq 1$ then
\begin{equation*}
B_{n,M}(x) = x^n-\frac{1}{M_{n+1}}\sum_{k=0}^{n-1}\binom{n+1}{k}_MB_{k,M}(x)
\end{equation*}
and
\begin{equation*}
E_{n,M}(x) = x^n-\frac{1}{2}\sum_{k=0}^{n-1}\binom{n}{k}_ME_{k,M}(x).
\end{equation*}
\end{corollary}
\begin{proof}
As \[B_{n,M}(x+_M1) = \sum_{k=0}^{n}\binom{n}{k}_MB_{k,M}(x)\]
and \[E_{n,M}(x+_M1) = \sum_{k=0}^{n}\binom{n}{k}_ME_{k,M}(x),\]
by the previous theorem we obtain 
\begin{align*}
\sum_{k=0}^{n+1}\binom{n+1}{k}_MB_{k,M}(x)-B_{n+1,M}(x) = M_{n+1} x^{n}.
\end{align*}
Hence, 
\[B_{n,k}(x) = x^n-\frac{1}{M_{n+1}}\sum_{k=0}^{n-1}\binom{n+1}{k}_MB_{k,M}(x),\]
with initial term $B_{0,M}(x) = 1.$
In the similar manner,
\[\sum_{k=0}^{n}\binom{n}{k}_ME_{k,M}(x)+ E_{n,M}(x) = 2x^n.\]
Hence,
\[E_{n,M}(x) = x^n-\frac{1}{2}\sum_{k=0}^{n-1}\binom{n}{k}_ME_{k,M}(x),\]
where $E_0(x) = 1.$
\end{proof}
As a consequence, we see that the sets $\{B_{0,M}(x), B_{1,M}(x), \dots B_{n,M}(x)\}$ as well as $\{E_{0,M}(x), \\E_{1,M}(x), \dots E_{n,M}(x)\}$ span the vector space of polynomials with degree at most $n.$
Further, taking $x=0$ yields the following recursions
\begin{equation}
B_{n,M} = -\frac{1}{M_{n+1}}\sum_{k=0}^{n-1}\binom{n+1}{k}_MB_{k,M}
\end{equation}
and 
\begin{equation}
E_{n,M} = -\frac{1}{2}\sum_{k=0}^{n-1}\binom{n}{k}_ME_{k,M}.
\end{equation}
Moreover, for $n \geq 2$ we obtain
\begin{equation} \label{eq: B1}
B_{n,M} = \sum_{k=0}^{n}\binom{n}{k}_MB_{k,M}
\end{equation}
and
\begin{equation} \label{eq: E1}
E_{n,M} = -\sum_{k=0}^{n}\binom{n}{k}_ME_{k,M}
\end{equation}
or 
\begin{equation} \label{eq: E2}
\sum_{k=0}^{n}\binom{n}{k}_ME_{k,M}+E_{n,M} = 2\delta_{n,0}.
\end{equation}
In the next result, we explicitly find the value of the functions $B_{n,M}(x)$ and $E_{n,M}(x)$ using the corresponding numbers.
\begin{theorem} \label{th: 3_1}
The functions $B_{n,M}(x)$ and $E_{n,M}(x)$ respectively satisfy the relations
\begin{equation*} 
B_{n,M}(x) = \sum_{k=0}^{n}\binom{n}{k}_MB_{k,M}x^{n-k}
\end{equation*}
and
\begin{equation*} 
E_{n,M}(x) = \sum_{k=0}^{n}\binom{n}{k}_ME_{k,M}x^{n-k}.
\end{equation*}
\end{theorem}
\begin{proof}
From the generating functions we obtain 
\begin{align*}
\sum_{n=0}^{\infty} B_{n,M}(x) \frac{t^n}{M_n!} &= \frac{t}{e^t_M-1}e^{tx}_M \\
&= \left(\sum_{n=0}^{\infty} B_{n,M} \frac{t^n}{M_n!}\right)\left(\sum_{m=0}^{\infty} x^m\frac{t^m}{M_m!}\right)\\
&= \sum_{n=0}^{\infty}\left(\sum_{k=0}^{m}\binom{n}{k}_MB_{n,M}x^{n-k}\right)\frac{t^n}{M_n!}.
\end{align*}
The result concerning Mersenne-Euler polynomials follows the same procedure.
\end{proof}
By the above theorem we conclude that the functions $B_{n,M}(x)$ and $E_{n,M}(x)$ are actually monic polynomials of degree $n,$ hence they are called Mersenne-Bernoulli and Mersenne-Euler polynomials respectively.
First few terms of the Mersenne-Bernoulli and Mersenne-Euler polynomials are
\[
B_{0,M}(x) = 1,
~B_{1,M}(x) = x-\frac{1}{3},
~B_{2,M}(x) = x^2-x+\frac{4}{21},
~B_{3,M}(x) = x^3-\frac{7}{3}x^2+\frac{4}{3}x-\frac{8}{45},\dots
\]
\[
E_{0,M}(x) = 1,
~E_{1,M}(x) = x-\frac{1}{2},
~E_{2,M}(x) = x^2-\frac{3}{2}x+\frac{1}{4},
~E_{3,M}(x) = x^3-\frac{7}{2}x^2+\frac{7}{4}x+\frac{3}{8}, \dots
\]
In the next two results we find the relations between these two class of polynomials.
\begin{theorem}
For $n \geq 1,$ we have
\begin{equation*}
E_{n-1,M}(x) = \frac{-2}{M_n}\sum_{k=1}^{n}\binom{n}{k}_ME_{k,M}B_{n-k,M}(x).
\end{equation*}
\end{theorem}
\begin{proof}
We see that
\begin{align*}
t\sum_{n=0}^{\infty}E_{n,M}(x)\frac{t^n}{M_n!}& = \frac{2te^{xt}_M}{e^t_M+1}\\
&= \left(e^t_M-1\right)\left(\frac{2}{e^t_M+1}\right)\left(\frac{te^{xt}_M}{e^t_M-1}\right)\\
&= 2\left(1-\frac{2}{e^t_M+1}\right)\left(\frac{te^{xt}_M}{e^t_M-1}\right)\\
&= 2\left(1-\sum_{n=0}^{\infty}E_{n,M}\frac{t^n}{M_n!}\right)\left(\sum_{l=0}^{\infty}B_{l,M}(x)\frac{t^l}{M_l!}\right)\\
&= \sum_{n=0}^{\infty}2\left(B_{n,M}(x)-\sum_{k=0}^{n}\binom{n}{k}_ME_{k,M}B_{n-k,M}(x)\right)\frac{t^n}{M_{n}!}\\
&= \sum_{k=1}^{\infty}\left(-2\sum_{k=1}^{n}\binom{n}{k}_ME_{k,M}B_{n-k,M}(x)\right)\frac{t^n}{M_n!}.
\end{align*}
However,
\begin{align*}
t\sum_{n=0}^{\infty}E_{n,M}(x)\frac{t^n}{M_n!} = \sum_{n=0}^{\infty}E_{n,M}(x)\frac{t^{n+1}}{M_n!}
= \sum_{n=1}^{\infty}M_{n}E_{n-1,M}(x)\frac{t^n}{M_n!}.
\end{align*}
Hence, we get the required result by equating the above two equations.
\end{proof}
\begin{theorem}
For $n \geq 1$ we have
\begin{equation*}
B_{n,M}(x) = \frac{M_nE_{n-1,M}}{2}+\sum_{k=0}^{n}\binom{n}{k}_MB_{k,M}E_{n-k,M}(x).
\end{equation*}
\end{theorem}
\begin{proof}
We observe that
\begin{align*}
\sum_{n=0}^{\infty}B_{n,M}(x)\frac{t^n}{M_n!} &= \frac{te^{xt}_M}{e^t_M-1}\\
&= \left(\frac{e^t_M+1}{2}\right)\left(\frac{t}{e^t_M-1}\right)\left(\frac{2e^{xt}_M}{e^t_M+1}\right)\\
&= \frac{1}{2}\left(t+\frac{2t}{e^t_M-1}\right)\left(\frac{2e^{xt}_M}{e^t_M+1}\right)\\
&=\left(\frac{t}{2}+ \sum_{n=0}^{\infty}B_{n,M}\frac{t^n}{M_n!}\right)\left(\sum_{l=0}^{\infty}E_{l,M}(x)\frac{t^l}{M_l!}\right)\\
&= \sum_{n=0}^{\infty}\frac{E_{n,M}}{2}\frac{t^{n+1}}{M_n!}+
\sum_{n=0}^{\infty}\sum_{k=0}^{n}\binom{n}{k}_MB_{k,M}E_{n-k,M}(x)\frac{t^{n}}{M_n!}\\
& = 1 + \sum_{n=1}^{\infty}\left(\frac{M_nE_{n-1,M}}{2}+\sum_{k=0}^{n}\binom{n}{k}_MB_{k,M}E_{n-k,M}(x)\right)\frac{t^{n}}{M_n!}.
\end{align*}
The theorem ends with comparing the coefficients.
\end{proof}
\begin{corollary}
If $n \geq 1,$
\begin{equation*}
\sum_{k=0}^{n}\binom{n}{k}_MB_{k,M}E_{n-k,M} = \frac{-M_nE_{n-1,M}}{2}+B_{n,M}.
\end{equation*}
\end{corollary}
\begin{proof}
Taking $x=0$ in the previous theorem proves this.
\end{proof}
We now turn towards more convolution like results of the Mersenne-Bernoulli and Mersenne-Euler numbers. For this purpose we define the operator $\mathcal{D}^t$ by
\begin{equation}
\mathcal{D}^t\left(f(t)\right)= 
\begin{cases}
0, & \text{~if~} t = 0\\
\frac{f(2t)-f(t)}{t}, & \text{~else.}
\end{cases}
\end{equation}
By simple computation, we can show that
\[\mathcal{D}^t\left(af(t)+bg(t)\right)= a\mathcal{D}^t\left(f(t)\right)+b\mathcal{D}^t\left(g(t)\right),\]
\[\mathcal{D}^t\left(f(t)g(t)\right)= g(2t)\mathcal{D}^t\left(f(t)\right)+f(t)\mathcal{D}^t\left(g(t)\right),\]
and for $f(t) \neq 0$
\[\mathcal{D}^t\left(\frac{1}{f(t)}\right) = \frac{-\mathcal{D}^t\left(f(t)\right)}{f(t)f(2t)}.\]
By direct calculation, we observe the following
\begin{itemize}
\item[i.]$\mathcal{D}^t(c) = 0$ for any constant $c.$
\item[ii.]$\mathcal{D}^t(t^n) = M_nt^{n-1}$ for $n \geq 1.$
\item[iii.]$\mathcal{D}^t(e^{xt}_M) = xe^{xt}_M.$
\end{itemize}
We note that the operator $\mathcal{D}^t$ is similar to the conventional derivative operator, we call it the \textbf{Mersenne-derivative} operator.
\begin{proposition} \label{prop:2_9}
For $n \geq 1$ we have 
\begin{equation*} \label{eq:DB}
\mathcal{D}^x(B_{n,M}(x)) = M_nB_{n-1,M}(x)
\end{equation*}
and
\begin{equation*} \label{eq:DE}
\mathcal{D}^x(E_{n,M}(x)) = M_nE_{n-1,M}(x).
\end{equation*}
\end{proposition}
\begin{proof}
We see that
\[\mathcal{D}^x\left(\frac{te^{xt}_M}{e^t_M-1}\right)= \sum_{n=0}^{\infty}\mathcal{D}^x(B_{n,M}(x))\frac{t^n}{M_n!}.\]
However,
\begin{align*}
\mathcal{D}^x\left(\frac{te^{xt}_M}{e^t_M-1}\right) &= \frac{t^2e^{xt}_M}{e^t_M-1}\\
&= \sum_{n=0}^{\infty}B_{n,M}(x)\frac{t^{n+1}}{M_n!}\\
&= \sum_{n=1}^{\infty}M_nB_{n-1,M}(x)\frac{t^{n}}{M_n!}.
\end{align*}
We obtain the required result from the above two equations.
The Mersenne-Euler polynomials follow the same method.
\end{proof}
\begin{theorem}
If $n \geq 1,$ then the Mersenne-Bernoulli numbers satisfy 
\begin{equation*}
\sum_{k=0}^{n}\binom{n}{k}_M 2^k B_{k,M}B_{n-k,M} = -2(M_{n-1}B_{n,M}+2^{n-2}M_nB_{n-1,M}).
\end{equation*}
\end{theorem}
\begin{proof}
Let $f(t) = \frac{t}{e^t_M-1},$ then
\begin{align*}
\mathcal{D}^t(f(t))&= \frac{1}{e^{2t}_M-1}+t\mathcal{D}^t\left(\frac{1}{e^t_M-1}\right)\notag\\
&= \frac{1}{e^{2t}_M-1}-\frac{te^t_M}{(e^t_M-1)(e^{2t}_M-1)}.
\end{align*}
Further,
\begin{align*} 
f(t)f(2t) &= \left(\frac{t}{e^t_M-1}\right)\left(\frac{2t}{e^{2t}_M-1}\right)\\
&= \frac{2t^2}{(e^t_M-1)(e^{2t}_M-1)}.
\end{align*}
It can be calculated that
\begin{equation} \label{eq:CBN}
f(t)f(2t) = (1-t)f(2t)-2t\mathcal{D}^t(f(t)).
\end{equation}
Now,
\begin{align*}
f(t)f(2t) &= \left(\sum_{n=0}^{\infty}B_{n,M}\frac{t^n}{M_n!}\right)\left(\sum_{l=0}^{\infty}2^lB_{l,M}\frac{t^l}{M_l!}\right)\\
&= \sum_{n=0}^{\infty}\left(\sum_{k=0}^{n}\binom{n}{k}_M2^kB_{k,M}B_{n-k,M}\right)\frac{t^n}{M_n!}
\end{align*}
whereas
\begin{align*}
(1-t)f(2t)-2t\mathcal{D}^t(f(t)) &= (1-t)\sum_{n=0}^{\infty}2^nB_{n,M}\frac{t^n}{M_n!}-2t\sum_{n=0}^{\infty}B_{n,M}\mathcal{D}^t\left(\frac{t^n}{M_n!}\right)\\
&= \sum_{n=0}^{\infty}2^nB_{n,M}\frac{t^n}{M_n!}-\sum_{n=0}^{\infty}2^nB_{n,M}\frac{t^{n+1}}{M_n!}-2t\sum_{n=1}^{\infty}B_{n,M}\frac{t^{n-1}}{M_{n-1}!}\\
&= \sum_{n=0}^{\infty}2^nB_{n,M}\frac{t^n}{M_n!}-\sum_{n=1}^{\infty}2^{n-1}M_nB_{n-1,M}\frac{t^n}{M_n!}-\sum_{n=1}^{\infty}2M_nB_{n,M}\frac{t^n}{M_n!}\\
&= 1+\sum_{n=1}^{\infty}(2^nB_{n,M}-2^{n-1}M_nB_{n-1,M}-2M_nB_{n,M})\frac{t^n}{M_n!}.
\end{align*}
Finally, by \eqref{eq:CBN} we obtain
\begin{align*}
\sum_{k=0}^{n}\binom{n}{k}_M2^kB_{k,M}B_{n-k,M}
&= 2^nB_{n,M}-2^{n-1}M_nB_{n-1,M}-2M_nB_{n,M}\\
&= -2(M_{n-1}B_{n,M}+2^{n-2}M_nB_{n-1,M})
\end{align*}
for $n \geq 1.$
\end{proof}
\begin{theorem}
For $n \geq 1$ the Mersenne-Euler numbers satisfy 
\begin{equation*}
\sum_{k=0}^{n}\binom{n}{k}_M2^kE_{k,M}E_{n-k,M}\ = 2(2^nE_{n,M}+E_{n+1,M}).
\end{equation*}
\end{theorem}
\begin{proof}
Let $g(t) = \frac{2}{e^t_M+1},$ then
\[g(t)g(2t) = \frac{4}{(e^t_M+1)(e^{2t}_M+1)}\]
and
\[\mathcal{D}^t(g(t)) = \frac{-2e^t_M}{(e^t_M+1)(e^{2t}_M+1)}.\]
Further, we note that
\begin{equation} \label{eq:CEN}
g(t)g(2t) = 2(g(2t) + \mathcal{D}^t(g(t))).
\end{equation}
Now,
\[g(t)g(2t) = \sum_{n=0}^{\infty}\left(\sum_{k=0}^{n}\binom{n}{k}_M2^kE_{k,M}E_{n-k,M}\right)\frac{t^n}{M_n!}\]
and
\begin{align*}
g(t)+\mathcal{D}^t(g(t)) &= \sum_{n=0}^{\infty}2^nE_{n,M}\frac{t^n}{M_n!}+\sum_{n=0}^{\infty}E_{n,M}\mathcal{D}^t\left(\frac{t^n}{M_n!}\right)\\
&= \sum_{n=0}^{\infty}2^nE_{n,M}\frac{t^n}{M_n!}+\sum_{n=1}^{\infty}E_{n,M}\frac{t^{n-1}}{M_{n-1}!}\\
&= \sum_{n=0}^{\infty}(2^nE_{n,M}+E_{n+1,M})\frac{t^n}{M_n!}.
\end{align*}
Hence, by \eqref{eq:CEN} we prove the theorem.
\end{proof}
Now, if $\mathcal{D}^x(f(x)) = F(x)$ we call $f(x)$ an anti Mersenne-derivative of $F(x)$ and define the \textbf{Mersenne-integral} as
\[\int_{a}^{b} F(x)~ d_M(x) = f(b)-f(a).\]
Moreover, for any constants $\alpha, \beta$ and functions $F(x) = \mathcal{D}^x(f(x)),~G(x) = \mathcal{D}^x(g(x))$ we know that
$\mathcal{D}^x(\alpha f(x)+\beta g(x)) = \alpha F(x) + \beta G(x).$ Therefore,
\[\int_{a}^{b} (\alpha F(x) + \beta G(x))~d_M(x) = \alpha(f(b)-f(a))+\beta(g(b)-g(a)).\]
Moreover,
\[\int_{a}^{b} \mathcal{D}^{x}(f(x)g(x))~d_M(x) = f(b)g(b)-f(a)g(a)\]
whereas 
\[\int_{a}^{b} \mathcal{D}^{x}(f(x)g(x))~d_M(x) = 
\int_{a}^{b} g(2x)\mathcal{D}^{x} (f(x)) ~d_M(x) +
\int_{a}^{b} f(x)\mathcal{D}^{x} (g(x))~d_M(x).\]
Thus, we obtain integration by parts like result as
\begin{equation}
\int_{a}^{b} f(x)\mathcal{D}^{x} (g(x))~d_M(x) = f(b)g(b)-f(a)g(a)-\int_{a}^{b} \mathcal{D}^{x} (f(x)) g(2x)~d_M(x)
\end{equation}
and 
\begin{equation} \label{eq:Int}
\int_{a}^{b} g(2x)\mathcal{D}^{x} (f(x)) ~d_M(x) = f(b)g(b)-f(a)g(a)-\int_{a}^{b} f(x)\mathcal{D}^{x} (g(x))~d_M(x).
\end{equation}
As an example we note that 
\begin{equation*}
\int_{0}^{1} x^n~ d_M(x) = \frac{1}{M_{n+1}}
\end{equation*}
for any non-negative $n.$ Further, the Mersenne-integrals of the Mersenne-Bernoulli and Mersenne-Euler polynomials are stated as the next result.
\begin{proposition}
For $n \geq 1$ the Mersenne-Bernoulli and Mersenne-Euler polynomials satisfy
\begin{equation*}
\int_{0}^{1} B_{n,M}(x)~d_M(x) = 0
\end{equation*}
and
\begin{equation*}
\int_{0}^{1} E_{n,M}(x)~d_M(x) = \frac{-2E_{n+1,M}}{M_{n+1}}.
\end{equation*}
\end{proposition}
\begin{proof}
The above can be proved using \eqref{eq:Bn1},\eqref{eq:En1}, Proposition \ref{prop:2_9}.
\end{proof}
For the next two results we use \eqref{eq:Int}.
\begin{theorem}
For $m,n\geq 1$
\begin{equation*}
\int_{0}^{1}x^nB_{m,M}(x)~d_M(x) = \frac{1}{2^n}\sum_{k=1}^{n}\frac{(-1)^{k-1}\binom{n+1}{k}_M}{M_{n+1}\binom{m+k}{k}_M}B_{m+k,M}.
\end{equation*}
\end{theorem}
\begin{proof}
Let $I_{m,n} = \int_{0}^{1}x^nB_{m,M}(x)~d_M(x).$  First we note that
\begin{align*} \label{eq:BI}
I_{m,1} &= \int_{0}^{1}xB_{m,M}(x)~d_M(x)\notag\\ 
&= \frac{1}{2} \int_{0}^{1}2x\mathcal{D}^x\left(\frac{B_{m+1,M}(x)}{M_{m+1}}\right)~d_M(x)\notag\\
&= \frac{B_{m+1,M}}{2M_{m+1}}.
\end{align*}
Then, we see that
\begin{align*}
I_{m,n} &= \frac{1}{2^n}\int_{0}^{1}(2x)^nB_{m,M}(x)~d_M(x)\\
&= \frac{1}{2^n}\int_{0}^{1}(2x)^n\mathcal{D}^x\left(\frac{B_{m+1,M}(x)}{M_{m+1}}\right)~d_M(x)\\
&= \frac{1}{2^n}\left(\frac{B_{m+1,M}}{M_{m+1}}-\frac{M_n}{M_{m+1}}\int_{0}^{1}x^{n-1}B_{m+1,M}(x)~d_M(x)\right)\\
&= \frac{1}{2^n}\left(\frac{B_{m+1,M}}{M_{m+1}}-\frac{M_n}{M_{m+1}}I_{m+1,n-1}\right)\\
&= \frac{1}{2^n}\left(\frac{B_{m+1,M}}{M_{m+1}}-\frac{1}{2^{n-1}}\left(\frac{M_nB_{m+2,M}}{M_{m+1}M_{m+2}}+\frac{M_{n}M_{n-1}}{M_{m+1}M_{m+2}}I_{m+2,n-2}\right)\right).
\end{align*}
Continuing this way we obtain
\begin{align*}
I_{m,n} &= \frac{B_{m+1,M}}{2^nM_{m+1}}-\frac{M_nB_{m+2,M}}{2^{2n-1}M_{m+1}M_{m+2}}+\dots+\frac{(-1)^{n-1}}{2^{\frac{n(n+1)}{2}-1}}\left(\prod_{j=1}^{n-1}\frac{M_{j+1}}{M_{m+j}}\right)I_{m+n-1,1}\\
&= \frac{B_{m+1,M}}{2^nM_{m+1}}-\frac{M_nB_{m+2,M}}{2^{2n-1}M_{m+1}M_{m+2}}+\dot+\frac{(-1)^{n-1}}{2^{\frac{n(n+1)}{2}}}\prod_{j=1}^{n}\frac{M_{j+1}}{M_{m+j}}\\
&= \frac{1}{M_{n+1}}\sum_{k=1}^{n}\frac{(-1)^{k-1}\binom{n+1}{k}_M}{2^{\frac{n(n+1)-(n-k)(n-k+1)}{2}}\binom{m+k}{k}_M}B_{m+k,M}.
\end{align*}
\end{proof}
\begin{theorem}
If $m,n \geq 1,$ we have
\begin{align*}
\int_{0}^{1}x^nE_{m,M}(x)~d_M(x) = \frac{1}{M_{n+1}}&\sum_{k=1}^{n}\frac{(-1)^{k}\binom{n+1}{k}_M}{2^{\frac{n(n+1)-(n-k)(n-k+1)}{2}}\binom{m+k}{k}_M}E_{m+k,M}\notag\\&
+\frac{(-1)^{n-1}}{2^{\frac{n(n+1)}{2}}}\frac{E_{m+n+1}}{M_{m+n+1}\binom{m+n}{n}_M}.
\end{align*}
\end{theorem}
\begin{proof}
We follow the similar method as the previous theorem. Let $J_{m,n} = \int_{0}^{1}x^nE_{m,M}(x)~d_M(x).$ Then we observe that
\begin{align*} 
J_{m,1} &= \int_{0}^{1}xE_{m,M}(x)~d_M(x)\notag\\ 
&= \frac{1}{2} \int_{0}^{1}2x\mathcal{D}^x\left(\frac{E_{m+1,M}(x)}{M_{m+1}}\right)~d_M(x)\notag\\
&= \frac{1}{2}\left(\frac{E_{m+1,M}}{M_{m+1}}-\frac{1}{M_{m+1}}\int_{0}^{1} E_{m+1,M}(x)~d_M(x)\right)\notag\\
&= \frac{E_{m+1,M}(1)}{2M_{m+1}}+\frac{E_{m+2,M}}{M_{m+1}M_{m+2}} \notag\\
&= \frac{-E_{m+1,M}}{2M_{m+1}}+\frac{E_{m+2,M}}{M_{m+1}M_{m+2}}.
\end{align*}
Therefore,
\begin{align*}
J_{m,n} &= \frac{1}{2^n}\int_{0}^{1} (2x)^n\mathcal{D}^x\left(\frac{E_{m+1,M}}{M_{m+1}}\right)\\
&= \frac{1}{2^n}\left(-\frac{E_{m+1,M}}{M_{m+1}}-\frac{M_n}{M_{m+1}}J_{m+1,n-1}\right)\\
&=\frac{-E_{m+1,M}}{2^nM_{m+1}}+\frac{M_nE_{m+2,M}}{2^{2n-1}M_{m+1}M_{m+2}}+\dots+\frac{(-1)^{n-1}}{2^{\frac{n(n+1)}{2}-1}}\left(\prod_{j=1}^{n-1}\frac{M_{j+1}}{M_{m+j}}\right)J_{m+n-1,1}\\
&= \frac{1}{M_{n+1}}\sum_{k=1}^{n}\frac{(-1)^{k}\binom{n+1}{k}_M}{2^{\frac{n(n+1)-(n-k)(n-k+1)}{2}}\binom{m+k}{k}_M}E_{m+k,M}+\frac{(-1)^{n-1}}{2^{\frac{n(n+1)}{2}}}\left(\prod_{j=1}^{n}\frac{M_{j+1}}{M_{m+j}}\right)\frac{E_{m+n+1,M}}{M_{m+n+1}}\\
&=  \frac{1}{M_{n+1}}\sum_{k=1}^{n}\frac{(-1)^{k}\binom{n+1}{k}_M}{2^{\frac{n(n+1)-(n-k)(n-k+1)}{2}}\binom{m+k}{k}_M}E_{m+k,M}+\frac{(-1)^{n-1}}{2^{\frac{n(n+1)}{2}}}\frac{E_{m+n+1}}{M_{m+n+1}\binom{m+n}{n}_M}.
\end{align*}
\end{proof}
The following two are miscellaneous results which we obtain by observing some pattern which are proved using \eqref{eq: MB} and induction on $n.$
\begin{theorem}
If the sequence $\{C_n\}_{n \geq 0}$ is defined as 
$C_n = -\frac{1}{M_n}\sum_{k=0}^{n-1}\binom{n+1}{k}_MC_k$
with $C_0 = 1,$ then
$C_n = (-1)^n 2^{\frac{n(n-1)}{2}}.$
\end{theorem}
\begin{proof}
We see that $n=0$ satisfies the above easily, hence assume the same for any non-negative integer strictly less than $n .$ The result follows by
\begin{align*}
-M_{n}C_n &= 1+ \sum_{k=1}^{n-1}\left(2^k\binom{n}{k}_M+\binom{n}{k-1}_M\right)C_k\\
&= \sum_{k=0}^{n-1}2^k\binom{n}{k}_MC_k +\sum_{k=1}^{n-1}\binom{n}{k-1}_MC_k\\
&= \sum_{k=0}^{n-1}2^k\binom{n}{k}_MC_k + \sum_{k=0}^{n-2}\binom{n}{k}_MC_{k+1}\\
&= \sum_{k=0}^{n-1}\binom{n}{k}_M (-1)^k ~2^{\frac{k(k+1)}{2}} + \sum_{k=0}^{n-2}\binom{n}{k}_M (-1)^{k+1}2^{\frac{k(k+1)}{2}}\\
&= (-1)^{n-1}M_{n}2^{\frac{n(n-1)}{2}}.
\end{align*}
\end{proof}
\begin{theorem}
If $\{D_n\}_{n\geq 0}$ is a sequence with 
$D_n = -\sum_{k=0}^{n-1}\binom{n}{k}_MD_k$
with $D_0 = 1,$
then $D_n = (-1)^n2^{\frac{n(n-1)}{2}}.$
\end{theorem}
\begin{proof}
We see that $n = 0$ satisfies the condition, hence assume $n > 0$ and the statement is true for for all non-negative integers less than it. Now,
\begin{align*}
-D_n &= \sum_{k=0}^{n-1}\binom{n}{k}_MD_k\\
&= 1+ \sum_{k=1}^{n-1}\binom{n}{k}_MD_k\\
&= 1+ \sum_{k=1}^{n-1}2^k\binom{n-1}{k}_MD_k + \sum_{k=1}^{n-1}\binom{n-1}{k-1}_MD_k\\
&= \sum_{k=0}^{n-1}2^k\binom{n-1}{k}_MD_k + \sum_{k=0}^{n-2}\binom{n-1}{k-1}_MD_{k+1}\\
&= \sum_{k=0}^{n-1}\binom{n-1}{k}_M(-1)^k 2^{\frac{k(k+1)}{2}}+
\sum_{k=0}^{n-2}\binom{n-1}{k}_M(-1)^{k+1} 2^{\frac{k(k+1)}{2}}\\
&= (-1)^{n-1}2^{\frac{n(n-1)}{2}}.
\end{align*}
The result follows by taking negative sign on both sides.
\end{proof}
\section{Mersenne-Bernoulli and Mersenne-Euler Matrices}
In this section we introduce square matrices of order $n\geq 2$ consisting of Mersenne-Bernoulli and Mersenne-Euler polynomials and find the corresponding inverses by factorising them. We demonstrate the obtained results using suitable examples.
Firstly, let the \textbf{Mersenne-Pascal} matrix $\mathcal{M}_n(x) = (a{ij})$ for $1\leq i, j \leq n$ be defined by
\begin{equation}
a{ij} = \begin{cases}
\binom{i-1}{j-1}_Mx^{i-j}, &\text{~if~} i \geq j\\
0, &\text{~else.}
\end{cases}
\end{equation}
By Theorem 1 in \cite{KT}, it can be shown that the inverse of the Mersenne-Pascal matrix $\mathcal{M}_n(x)^{-1} = (b{ij})$ where
\begin{equation}
b{ij} = \begin{cases}
c_{i-j+1}\binom{i-1}{j-1}_Mx^{i-j}, &\text{~if~} i \geq j\\
0, &\text{~else}
\end{cases}
\end{equation}
with $c_1 = 1$ and $c_n = -\sum_{k=1}^{n-1}\binom{n-1}{k-1}_Mc_k.$
We now define the \textbf{Mersenne-Bernoulli} matrix $\mathcal{B}_n(x)=(u{ij})$ for $1 \leq i,j \leq n$ by
\begin{equation}
u{ij} = \begin{cases}
\binom{i-1}{j-1}_MB_{i-j,M}(x), &\text{~if~} i \geq j\\
0, &\text{~else}
\end{cases}
\end{equation}
and the \textbf{Mersenne-Euler} matrix $\mathcal{E}_n(x) = (v{ij})$
for $1 \leq i,j \leq n$ is stated as 
\begin{equation}
v{ij} = \begin{cases}
\binom{i-1}{j-1}_ME_{i-j,M}(x), &\text{~if~} i \geq j\\
0, &\text{~else.}
\end{cases}
\end{equation}
\begin{theorem}
Let $\mathcal{B}_n = \mathcal{B}_n(0)$ and
$\mathcal{Q}_n = (q{ij})$ be a square matrix with
\[q_{ij} = \begin{cases}
\frac{1}{M_{i-j+1}}\binom{i-1}{j-1}_M, &\text{~if~} i \geq j\\
0, &\text{~else}
\end{cases}\]
for $1 \leq i,j \leq n.$ Then $\mathcal{B}_n^{-1} = \mathcal{Q}_n.$
\end{theorem}
\begin{proof}
We shall make use of \eqref{BN: 1}. Now, suppose $\mathcal{R}_n = \mathcal{Q}_n\mathcal{B}_n = (r{ij}),$ then 
\begin{align*}
r{ij} &= \sum_{k=1}^{n}q_{ik}b_{kj} = \sum_{k=j}^{i}q_{ik}b_{kj}\\
&= \sum_{k=j}^{i} \binom{i-1}{k-1}_M \binom{k-1}{j-1}_M\frac{1}{M_{k-j+1}}B_{i-k,M}\\
&= \binom{i-1}{j-1}_M\sum_{k=j}^{i} \binom{i-j}{k-j}_M\frac{1}{M_{k-j+1}} B_{i-k,M}\\
&= \binom{i-1}{j-1}_M\sum_{k=0}^{i-j} \binom{i-j}{k}_M\frac{1}{M_{k+1}} B_{i-j-k,M}\\
&= \binom{i-1}{j-1}_M M_{i-j}!\delta_{i-j,0} = \delta_{i,j}.
\end{align*}
Therefore, $\mathcal{B}_n^{-1} = \mathcal{Q}_n.$
\end{proof}
\begin{theorem}
If $I_n$ is the identity matrix of order $n$ and $\mathcal{E}_n = \mathcal{E}_n(0)$ then $\mathcal{E}_n^{-1} = \frac{1}{2}(\mathcal{M}_n+I_n).$
\end{theorem}
\begin{proof}
By \eqref{eq: E2} and direct calculation
\begin{align*}
\left(\mathcal{E}_n\left(\mathcal{M}_n+I_n\right)\right)_{ij} &= \sum_{k=j}^{i} \left(\binom{i-1}{k-1}_M\binom{k-1}{j-1}_ME_{i-k,M}+\binom{i-1}{j-1}_ME_{i-j,M}\right)\\
&= \binom{i-1}{j-1}_M\sum_{k=j}^{i}\left(\binom{i-j}{k-j}_ME_{i-k,M}+E_{i-j,M}\right)\\
&= \binom{i-1}{j-1}_M\sum_{k=0}^{i-j}\left(\binom{i-j}{k}_ME_{i-j-k,M}+E_{i-j,M}\right)\\
&= 2\binom{i-1}{j-1}_M\delta_{i-j,0}= 2\delta_{i,j}.
\end{align*}
\end{proof}
\begin{theorem}
The Mersenne-Bernoulli and Mersenne-Euler polynomials matrices
satisfy the product formula
\begin{equation}
\mathcal{B}_n(x+_My) = \mathcal{B}_n(x)\mathcal{M}_n(y)
\end{equation}
and 
\begin{equation}
\mathcal{E}_n(x+_My) = \mathcal{E}_n(x)\mathcal{M}_n(y).
\end{equation}
\end{theorem}
\begin{proof}
We shall use Theorem \ref{TH: BE} to prove this theorem. 
If $i \geq j,$ we see that
\begin{align*}
\left(\mathcal{B}_n(x+_My)\right)_{ij} &= \binom{i-1}{j-1}_MB_{M,i-j}(x+_My)\\
&= \binom{i-1}{j-1}_M\sum_{k=0}^{i-j}\binom{i-j}{k}B_{k,M}(x)y^{i-j-k}.\\
\end{align*}
Further, for $i \geq j$
\begin{align*}
\left(\mathcal{M}_n(y)\mathcal{B}_n(x)\right)_{ij}&= \sum_{k=j}^{i}a_{ik}b_{kj}\\
&= \sum_{k=j}^{i}\binom{i-1}{k-1}_M\binom{k-1}{j-1}_My^{i-k}B_{k-j,M}(x)\\
&= \binom{i-1}{j-1}_M\sum_{k=j}^{i}\binom{i-j}{k-j}_My^{i-k}B_{k-j,M}(x)\\
&= \binom{i-1}{j-1}_M\sum_{k=0}^{i-j}\binom{i-j}{k}_MB_{k,M}(x)y^{i-j-k}.\\
\end{align*}
Thus, $\left(\mathcal{B}_n(x+_My)\right)_{ij} = \left(\mathcal{M}_n(y)\mathcal{B}_n(x)\right)_{ij}$ for $i \geq j.$ Now, if $i < j$
then we observe that $\left(\mathcal{B}_n(x+_My)\right)_{ij} = 0 = \left(\mathcal{M}_n(y)\mathcal{B}_n(x)\right)_{ij}.$
The result concerning the Mersenne-Euler matrix follows the same method.
\end{proof}
\begin{corollary}
The Mersenne-Bernoulli and Mersenne-Euler matrices satisfy 
\[\mathcal{B}_n(x)^{-1} = \mathcal{M}_n(x)^{-1}\mathcal{B}_n^{-1} \text{~and~} \mathcal{E}_n(x)^{-1} = \mathcal{M}_n(x)^{-1}\mathcal{E}_n^{-1}.\]
\end{corollary}
\begin{proof}
Replacing $x$ by $y$ and taking $y = 0$ in the previous theorem prove this result.
\end{proof}
\noindent The Mersenne-Bernoulli matrix of order $3$ and its inverse can be calculated to be
\[\mathcal{B}_3(x) = 
\begin{bmatrix}
1 & 0 & 0\\
x-\frac{1}{3} & 1 & 0\\
x^2 - x + \frac{4}{21} & 3x - 1 & 1\\
\end{bmatrix} \text{~and~}
\mathcal{B}_3(x)^{-1} =
\begin{bmatrix}
1 & 0 & 0 \\
-x+\frac{1}{3} & 1 & 0 \\
2x^2-x+\frac{1}{7} & -3x+1 & 1
\end{bmatrix}.\]
Similarly, the Mersenne-Euler matrix of order $3$ and its inverse are
\[\mathcal{E}_3(x) =
\begin{bmatrix}
1 & 0 & 0 \\
x-\frac{1}{2} & 1 & 0 \\
x^2-\frac{3}{2}x+\frac{1}{4} & 3x-\frac{3}{2} & 1 
\end{bmatrix}
\text{~and~}
\mathcal{E}_3(x)^{-1} =
\begin{bmatrix}
1 & 0 & 0\\
-x+\frac{1}{2} & 1 & 0\\
2x^2-\frac{3}{2}x+\frac{1}{2} & -3x+\frac{3}{2} & 1
\end{bmatrix}.
\]
\section{Conclusion}
In this work, we have introduced and systematically investigated the Mersenne-Bernoulli and Mersenne-Euler polynomials, establishing their fundamental properties, generating functions, and explicit identities. Our results extend the theory of special polynomials connected to Mersenne numbers and provide a new avenue for research in number theory and combinatorics. A significant direction for further study is to explore these polynomials from a probabilistic perspective. Following the modern approach demonstrated for degenerate polynomials by Kim and Kim \cite{KK1,KK2}, a key objective would be to construct associated random variables whose moments are given by our polynomials. This probabilistic interpretation would not only offer a unifying framework but could also facilitate the discovery of new results such as Spivey-type recurrence relations similar to those found for degenerate Bell and Dowling polynomials \cite{KK3,KK4}. Such an investigation would undoubtedly deepen the combinatorial and probabilistic significance of the structures introduced here.

\end{document}